\documentclass[11pt]{amsart}
\usepackage{amssymb,amscd,amsmath,enumerate}
\usepackage[all]{xy}
\usepackage[utf8]{inputenc}
\usepackage[english]{babel}
\usepackage{tikz}
\usetikzlibrary{positioning,chains,fit,shapes,calc}
\usepackage{mathtools}
\usepackage[colorlinks,linkcolor=blue]{hyperref}
\usepackage{ytableau}

\newtheorem{thm}{Theorem}

\newtheorem{cor}[thm]{Corollary}

\newtheorem{prop}[thm]{Proposition}

\newtheorem{lem}[thm]{Lemma}
\theoremstyle{remark}
\newtheorem{remark}[thm]{Remark}
\newtheorem{exam}[thm]{Example}

\theoremstyle{definition}
\newtheorem{defn}[thm]{Definition}

\renewcommand{\bar}{\overline}

\newcommand{\Ind}{{\rm Ind}}
\newcommand{\Irr}{{\rm Irr}\,}

\newcommand{\R}{{\mathbb{R}}}
\newcommand{\C}{{\mathbb{C}}}

\newcommand{\Q}{{\mathbb{Q}}}
\newcommand{\Z}{{\mathbb{Z}}}

\newcommand{\cG}{{\mathcal{G}}}
\newcommand{\cC}{{\mathcal{C}}}

\newcommand{\Aut}{\mathrm{Aut}\,}

\newcommand{\Sym}{\mathrm{Sym}}

\newcommand{\Alt}{{\raise 2pt\hbox{$\scriptstyle\bigwedge$}}}

\definecolor{myblue}{RGB}{80,80,160}
\definecolor{mygreen}{RGB}{80,160,80}
\newdimen\nodeSize
\newdimen\nodeDist
\tikzset{
	position/.style args={#1:#2 from #3}{
		at=(#3.#1), anchor=#1+180, shift=(#1:#2)
	}
}
\title{On counting centralizer subgroups of symmetric groups}
	\author{Zhipeng Lu}
	\address{Mathematisches Institut\\
		Georg-August Universit\"{a}t G\"{o}ttingen\\
		Bunsenstra\ss e 3-5\\
		D-37073 G\"{o}ttingen\\
		Germany}
	\email{Zhipeng.Lu@uni-goettingen.de}
\begin{document}
	\maketitle
	\section*{Introduction}\label{section-introduction}
	The paper is motivated by a MathOverflow question of Harald Helfgott \cite{mathoverflow}, which relates to growth in transitive permutation groups in his work \cite{Helfgott-St. Andews}. For any positive integer $m$, let $S_{2m}$ be the symmetric group on the symbols $\{1,2,\cdots,2m\}$ and $H$ be the subgroup of $S_{2m}$ consisting of permutations preserving the partition $\{1,2\},\{3,4\},\cdots,\{2m-1,2m\}$, or equivalently $H=C((1\ 2)(3\ 4)\cdots(2m-1\ 2m))$, the centralizer subgroup. Call $g\in S_{2m}$ \textit{good} if $|H\cap gHg^{-1}|=m^{O(1)}$; call it \textit{bad} otherwise. Then Helfgott wonders about the structure of good elements and postulated that the good permutations have density $1$ in $S_{2m}$. There seems to be a fair share of good permutations in $S_{2m}$, for example if the cycle decomposition of $g$ does not contain ``too many" cycles of the same length, $g$ may be checked good. The paper contributes to studying the structure of good elements, and shows that the good permutations have density zero as a negative answer to Helfgott's postulation.
	
	To proceed, we first clarify the structure of $H\cap gHg^{-1}$ for any $g\in S_{2m}$ in section \ref{section-structures of double cosets}. It turns out that the isomorphism class of $H\cap gHg^{-1}$ depends on the double coset $HgH$ and moreover
	\begin{thm}\label{thm-structure theorem}
		 Each $HgH$ has a representative $x\in\Sym \{2,4,\cdots,2m\}$ $\leq S_{2m}$ determined by a partition of $m$ and 
		  there is a one-to-one correspondence $H\backslash S_{2m}/H\leftrightarrow\{\text{parititions of }m\}$. Furthermore,
		for any $g\in HxH$ with $x\in\Sym\{2,4,\cdots,2m\}$ whose cycle decomposition has $r_i$ cycles of length $i$, $i=1,\cdots,k$,
		\[H\cap gHg^{-1}\simeq \bigoplus_{i=1}^kD_{i}\wr S_{r_i},\]
		where $D_i$ is the dihedral group with $2i$ elements. (For convenience we write $D_1$ for $C_2$ or $S_2$.)
	\end{thm}
		
	Thus $|H\cap gHg^{-1}|$ can be seen as a random variable on partitions of $m$ with probability distribution of counting measure $P(\lambda)=\frac{|HxH|}{|S_{2m}|}$, if $g\in HxH$ for $x\in\Sym\{2,4,\cdots,2m\}$ with cycle type $\lambda$. Then we prove
	\begin{thm}\label{thm-left tail}For any $c>0$,
		\[P(|H\cap gHg^{-1}|<m^c)\rightarrow 0, \text{ as }m\rightarrow\infty.\]
		Consequently, good elements of $S_{2m}$ have density zero.
	\end{thm}
	The right tail of $P$ is also estimated to show that
    \begin{thm}\label{thm-right tail}For some constant $C>0$,
    	\[P\left(|H\cap gHg^{-1}|>Cm^{\log m}\right)\rightarrow 0, \text{ as }m\rightarrow\infty.\]
    	In particular, the bad elements $g\in S_{2m}$ with $|H\cap gHg^{-1}|\gg m^{\log m}$ have zero density.
    \end{thm}

    
    \emph{Outline of paper}. 
    
    The 1-1 correspondence $H\backslash S_{2m}/H\leftrightarrow\{\text{parititions of }m\}$ of Theorem \ref{thm-structure theorem} is established inductively by studying the left and right action of $H$ on $S_{2m}$ in details in section \ref{subsection-double cosets counting}. It can also be verified by a character formula in section \ref{subsection-character formulas}. Then combined with an idea of bipartite graph automorphism construction introduced by J. P. James \cite{J.P.James}, we prove the structure result in Theorem \ref{thm-structure theorem} in section \ref{subsection-finer structure of centralizers}. As a byproduct we prove that they are all rational groups in aspect of representation theory. Section \ref{subsection-structure theorem verification} gives some computational verification of Theorem \ref{thm-structure theorem}.    
    
    Explicitly shown in section \ref{subsection-distribution of size of centralizers}, the distribution of $|H\cap gHg^{-1}|$ happens to be $P=\rm ESF(\frac{1}{2})$, where $\rm ESF(\frac{1}{2})$ is the Ewens' distribution with bias $\frac{1}{2}$. Then we estimate the left tail $P(\leq m^c)$ by the moment bound. The expectations for each $m$ involved in the moment bound are assembled into a special generating function. Then asymptotics of the expectations can be extracted from coefficients of singular expressions of the generating function around its singularities which are of logarithmic type, see section \ref{subsection-exp-log schema}. We use techniques from analytic combinatorics, especially the hybrid method introduced by Flajolet et al \cite{Flajolet}, to find the correct asymptotics and prove Theorem \ref{thm-left tail} in section \ref{subsection-hybrid method}.
    
    Following the same probabilistic setting, the expectations involved in the moment bound of the right tail are assembled in to generating functions with singularities of exponential type. Then to prove Theorem \ref{thm-right tail}, we use asymptotics of coefficients of generating functions of exponential type which was given by E. M. Wright \cite{Wright}, in section \ref{section-Tail of Ewens' distribution}.
    \medskip\medskip
    
	\textbf{Acknowledgement}\quad The author is supported by ERC Consolidator grant 648329 (GRANT), leaded by Professor Harald Helfgott. The author owes gratitute to H. H. ($=H^2$) for introducing the subject and bringing up the problem. The author also thanks Lifan Guan for helpful discussion and suggestions on editing the paper. 
	\section{Structures of double cosets and $H\cap gHg^{-1}$}\label{section-structures of double cosets}
	\subsection{Preliminaries on $H$ and $H\cap gHg^{-1}$}\label{section-preliminaries and examples}
	This section includes some necessary basic group theoretic results on $H=C(h_0)\leq S_{2m}$ for $h_0=(1\ 2)(3\ 4)\cdots (2m-1\ 2m)$ and $H\cap gHg^{-1}$ for general $g\in S_{2m}$.

	Firstly, viewed as preserving the block partition $\{1,2\}, \{3,4\},\cdots, \{2m-1,2m\}$ of $1,2\cdots,2m$, the structure of $H$ is as simple as follows
	\begin{prop}\label{prop-structure of H}
		$H$ has the wreath product structure $H=C(h_0)\simeq C_2\wr S_m$.
	\end{prop}
	This is also an easy corollary of 4.1.19 
	of James-Kerber \cite{James-Kerber} which describes the centralizer of any permutation in a symmetric group as a wreath product of cyclic groups with smaller symmetric groups. 
	One immediately notices that $H\cap gHg^{-1}$ is identical for any $g$ in a common left coset of $H$. Moreover, for any $h_1,h_2\in H$ and $g\in S_{2m}$,
	\[H\cap h_1gh_2H(h_1gh_2)^{-1}=H\cap h_1gHg^{-1}h_1^{-1}=h_1\left(H\cap gHg^{-1}\right)h_1^{-1},\]
	hence the structure of $H\cap gHg^{-1}$ depends only on the double coset $HgH$. 
	
    \begin{exam}\label{exam-m=2}
    Let $m=2$, then $H=D_1\wr S_2\simeq(C_2)^2\rtimes S_2$ and $S_4/H=\{\bar{1},\bar{(1\ 3)}, \bar{(1\ 4)}\}$. Computing by hand we get 
    \[H\cap (1\ 3)H(1\ 3)=\{1, (1\ 2)(3\ 4),(1\ 3)(2\ 4),(1\ 4)(2\ 3)\}=K_4,\]
    where $K_4$ is the Klein four group. Again by hand 
    \[H\cap (1\ 4)H(1\ 4)=H\cap (1\ 3)H(1\ 3)=K_4.\]
    This is no wonder because there are only $2$ double cosets in $H\backslash S_4/H$ with representatives $1$ and $(1\ 3)$, and clearly $(3\ 4)(1\ 3)(3\ 4)=(1\ 4)$ (note that $(3\ 4)\in H$).
    \end{exam} 
	\subsection{Double coset decomposition of $S_{2m}$}\label{subsection-double coset decomposition}
	Counting the left cosets contained in $HgH$ gives
	\[|HgH|=|H|[H:H\cap gHg^{-1}]=\dfrac{|H|^2}{|H\cap gHg^{-1}|}.\]
	Thus if each double coset determines a distinct structure (or size) of $H\cap gHg^{-1}$, the density of those $g$ is assigned by 
	\[\dfrac{|HgH|}{|S_{2m}|}=\dfrac{|H|^2}{|S_{2m}||H\cap gHg^{-1}|}=\dfrac{(2^mm!)^2}{(2m)!|H\cap gHg^{-1}|}.\] 
	In addition, the double coset decomposition of $S_{2m}$ by $H$ gives
	\[\label{equation-double coset decomposition}|S_{2m}|=\sum_{g\in H\backslash S_{2m}/H}|HgH|=\sum_{g\in H\backslash S_{2m}/H}\dfrac{|H|^2}{|H\cap gHg^{-1}|},\]	
	and consequently
	\[\sum_{g\in H\backslash S_{2m}/H}\dfrac{1}{|H\cap gHg^{-1}|}=\dfrac{|S_{2m}|}{|H|^2}=\dfrac{(2^mm!)^2}{(2m)!}\sim \dfrac{1}{\sqrt{\pi m}},\]
	by Stirling's formula. These formulas become the starting point of studying distribution of $|H\cap gHg^{-1}|$ in section \ref{section-counting good elements}.
	 
	\subsection{Counting double cosets by partition number}\label{subsection-double cosets counting}
	To describe the structure of $H\backslash S_{2m}/H$, we first prove the following lemma on double coset representatives.
	\begin{lem}\label{lem-double coset rep's}
		Each double coset of $H\backslash S_{2m}/H$ has a representative supported on the odd integers $M=\{1,3,\cdots,2m-1\}$ or the even integers $M'=\{2,4,\cdots,2m\}$.
	\end{lem}
	\begin{proof}
	We use induction on $m$. $m=1$ is trivial. Suppose the lemma is true for any $m'<m$. For any $x\in S_{2m}$, if $x=yz$ with $y$ and $z$ supported on the $2$-blocks $\coprod_{t\in N}\{t,t+1\}$ and $\coprod_{t\notin N}\{t,t+1\}$ respectively for some proper subset $N\subset M$ (in particular $y$ commutes with $z$), then by induction $y$ and $z$ can be made into permutations on $N$ and $M\smallsetminus N$ through multiplying on left and right by $H$ restricted to $\coprod_{t\in N}\{t,t+1\}$ and $\coprod_{t\notin N}\{t,t+1\}$ respectively. Hence $HxH$ has a representative supported on $M$. 
	
	Otherwise, in the cycle decomposition of $x$ every cycle shares supports on some block $\{2k-1, 2k\}$ with another cycle. 
	If a cycle in $x$ contains both $2k-1$ and $2k$, then it can be written as
	\[(2k-1\ l_1\cdots\ l_s\ 2k\ l'_1\ \cdots\ l'_t)\]
	\[=(2k-1\ l_1)\cdots(2k-1\ l_s)(2k-1\ 2k)(2k-1\ l'_1)\cdots(2k-1\ l'_t),\]
	with all numbers distinct.
	Multiplying $(2k-1\ 2k)$ ($\in H$) on the left on both sides above we get
	\[(2k-1\ 2k)x=(2k\ l_1\ \cdots\ l_s)(2k-1\ l'_1\ \cdots\ l'_t)\cdots,\]
	i.e. we can decompose the cycle into two cycles which split $\{2k-1,2k\}$. Repeat the procedure using suitable $(2k_i-1\ 2k_i)$ ($\in H$) $i=1,\cdots,r$, until $(2k_1-1\ 2k_1)\cdots(2k_r-1\ 2k_r)x$ has no cycles containing any $\{2k-1, 2k\}$. (This is doable since $(2k_i-1\ 2k_i)$ commutes with the cycles not intersecting $\{2k_i-1, 2k_i\}$.) 
	
	For a representative with such cycle type, by multiplying $(2k-1\ 2k)$'s simultaneously on left and right, we get a product of cycles which contains only either odd numbers or even numbers. Then move all cycles of even numbers to the left and by commutativity replace them by corresponding cycles of the complementary odd numbers by multiplying on the left the unique element in $H$ supported on the corresponding $2$-blocks. (For example, $(2\ 6\ 4)(8\ 10)$ can be replaced by $(1\ 5\ 3)(7\ 9)$ since $(2\ 6\ 4)(8\ 10)(1\ 5\ 3)(7\ 9)=(2\ 6\ 4)(1\ 5\ 3)(8\ 10)(7\ 9)\in H$.)
	Thus we get a representative of $HxH$ supported on odd numbers. Replacing the cycles of odd number by complementary even numbers we get a representatives supported on even integers.
	\end{proof}
	In addition, the following explicit expression of Proposition \ref{prop-structure of H} is crucial to proving the main result of this section.
	\begin{lem}\label{lem-TC decomposition}
		Let $M=\{1,3,\cdots,2m-1\},M'=\{2,4,\cdots,2m\}$, $C=\quad$ $\prod_{i=1}^m \Sym\{2i-1, 2i\}\le S_{2m}$, and $T=\Sym\{(1,2),\cdots,(2m-1, 2m)\}\le S_{2m}$ (the symmetric group of the ordered pairs $(2k-1,2k)$'s). Then $H=TC$ and explicitly for any $h\in H$, there is a unique decomposition \[h=\bar{h}\tilde{h}=\bar{h}_M\bar{h}_{M'}\tilde{h}=\bar{h}_{M'}\bar{h}_M\tilde{h},\] 
		in which $\tilde{h}\in C$, $\bar{h}\in T$, $\bar{h}_M$ and $\bar{h}_{M'}$, commuting with each other, are the complementary permutation actions of $\bar{h}$ restricted onto $M$ and $M'$ respectively. We call it the TC-decomposition of $H$.
	\end{lem}
	\begin{proof}
		For any $h\in H$ and $k\leq m$, let $\bar{h}$ be the permutation action defined as 
		\[\bar{h}\cdot(2k)=\begin{cases}h\cdot(2k), \text{ if } h\cdot(2k) \text{ is even},\\
		h\cdot(2k-1), \text{ if } h\cdot(2k) \text{ is odd},
		\end{cases}\]
		and
		$$\bar{h}\cdot(2k-1)=\begin{cases}h\cdot(2k-1), \text{ if } h\cdot(2k) \text{ is odd},\\
		h\cdot(2k), \text{ if } h\cdot(2k) \text{ is even},
		\end{cases}$$
		where $h\cdot i$ denotes the number that $h$ moves $i$ to. 
		
		The definition guarantees that $\bar{h}$ sends even numbers to even numbers and odd to odd while still preserving the partition $\{1,2\},\cdots,\{2m-1,2m\}$, hence belongs to $H$. The case separation in the definition where $2k-1$ and $2k$ are switched by $h$ gives a product of transpositions $(2k-1,2k)$'s, denoted by $\tilde{h}$. This amounts to the decomposition $h=\bar{h}\tilde{h}$ which is unique simply because $C\cap T=\{1\}$. Restriction of $\bar{h}$ onto $M$ and $M'$ gives the $3$-term decomposition
		\[h=\bar{h}_M\bar{h}_{M'}\tilde{h}=\bar{h}_{M'}\bar{h}_M\tilde{h},\]
		whose uniqueness is due to the decomposition $T=\Sym{M}\times \Sym(M')$.
	\end{proof}
	\begin{remark}\label{remark-CT decomposition}
		Note that alternatively we have the CT-decomposition $H=CT$, i.e. $h=\tilde{h}'\bar{h}$ for some $\tilde{h}'\in C$ which switches $h\cdot(2k)$ and $h\cdot(2k-1)$ when necessary. 
	\end{remark}
	Now we can prove the main result on the structure of $H\backslash S_{2m}/H$.
	\begin{prop}\label{prop-structure of double cosets}
		Keep the notations from last lemma. Each conjugacy class of $\Sym(M)$ ($\Sym(M')$) is contained in a distinct double coset of $H\backslash S_{2m}/H$, and each double coset intersects with $\Sym(M)$ ($\Sym(M')$) at a conjugacy class of $\Sym(M)$ ($\Sym(M')$). Consequently, $|H\backslash S_{2m}/H|=p(m)$, the partition number of $m$.
	\end{prop}
	\begin{proof}
		For any two conjugates $x_1,x_2\in \Sym(M)$, say conjugated by $x=(2k_1-1\ 2k_2-1\ \cdots\ 2k_s-1)\cdots (2k'_1-1\ 2k'_2-1\ \cdots\ 2k'_t-1)$, then they are conjugate in $S_{2m}$ by $x'=x(2k_1\ 2k_2\ \cdots\ 2k_s)\cdots(2k'_1\ 2k'_2\ \cdots\ 2k'_t)\in H$. 
		Hence $x_2\in Hx_1H$.

		On the other hand, if $x_2\in Hx_1H$, then there exists $h\in H$ such that $x_1hx_2^{-1}\in H$. By Lemma \ref{lem-TC decomposition} we get
		\begin{equation}\label{equation-TC decomposition}x_1hx_2^{-1}=x_1\bar{h}_M\bar{h}_{M'}\tilde{h}x_2^{-1}=x_1\bar{h}_M\bar{h}_{M'}(\tilde{h}x_2^{-1}\tilde{h})\tilde{h}.\end{equation}
		It is easy to check that $chc^{-1}=chc\in\Sym(M)$ ($\Sym(M')$) for any $h\in\Sym(M)$ ($\Sym(M')$) and any $c\in C$ such that $h$ preserves the support of $c$, denoted by $\mathrm{supp}(c)$. 
		
		We claim that $x_2^{-1}$ preserves $\mathrm{supp}(\tilde{h})$. For any $k\le m$, if $2k-1\notin \mathrm{supp}(\tilde{h})$, then $x_1hx_2^{-1}\cdot (2k)=x_1h\cdot (2k)=h\cdot (2k)$ is even.  Since $x_1hx_2^{-1}\in H$, $x_1hx_2^{-1}\cdot (2k-1)$ must be odd, which indicates $x_2^{-1}\cdot (2k-1)\notin \mathrm{supp}(\tilde{h})$. If $2k-1\in \mathrm{supp}(\tilde{h})$, then $x_1hx_2^{-1}\cdot (2k)=x_1h\cdot(2k)$ is odd. Hence $x_1hx_2^{-1}\cdot (2k-1)=hx_2^{-1}\cdot (2k-1)$ is even, and $x_2^{-1}\cdot (2k-1)\in \mathrm{supp}(\tilde{h})$. This shows
		\[x_2^{-1}(M\smallsetminus \mathrm{supp}(\tilde{h}))=M\smallsetminus \mathrm{supp}(\tilde{h}),\ x_2^{-1}(M\cap \mathrm{supp}(\tilde{h}))=M\cap \mathrm{supp}(\tilde{h}),\]
		and consequently $\tilde{h}x_2^{-1}\tilde{h}\in\Sym(M)$.
		
		Therefore in (\ref{equation-TC decomposition}), we can switch $\bar{h}_{M'}$ and $(\tilde{h}x_2^{-1}\tilde{h})$ to get
		\[x_1hx_2^{-1}=\left(x_1\bar{h}_M(\tilde{h}x_2^{-1}\tilde{h})\right)\bar{h}_{M'}\tilde{h},\]
		which must be the $3$-term TC-decomposition of $h$. Hence $x_1hx_2^{-1}=h$ and
		\[x_1\bar{h}_M(\tilde{h}x_2^{-1}\tilde{h})=\bar{h}_M,\]
		i.e. $x_1$ is conjugate to $\tilde{h}x_2\tilde{h}$ by $\bar{h}_M\in\Sym(M)$. 
		
		Furthermore we can choose $h\in H$ with $\tilde{h}=1$, so that $x_1$ is conjugate to $x_2$ by $\bar{h}_M\in\Sym(M)$. Actually since $x_2^{-1}$ preserves $\mathrm{supp}(\tilde{h})$, it is easy to verify that $x_2\tilde{h}x_2^{-1}\in H$. Then the TC-decomposition $x_1hx_2^{-1}=x_1\bar{h}x_2^{-1}x_2\tilde{h}x_2^{-1}$ ($\in H$) implies $x_1\bar{h}x_2^{-1}\in H$. Thus we can choose $h\in T$ in the beginning.
		
		Finally, since $\Sym(M)\simeq S_m$ in an obvious way, the conjugacy classes of $\Sym(M)$ hence the double cosets $H\backslash S_{2m}/H$ are in one-to-one correspondence with the partitions of $m$.
	\end{proof}
	\begin{remark}

		Note that if $x$ preserves the support of $c\in C$, then $cxc$ is the \textit{truncation} of $x$ from $\mathrm{supp}(c)$, i.e. $cxc\mid_{\mathrm{supp}(c)}$ is the trivial permutation and $cxc$ is the same permutation as $x$ outside of $\mathrm{supp}(c)$.
	\end{remark}
	\begin{remark}
		Now we can show by Stirling's formula that
		\[Average\ of\ |H\cap gHg^{-1}|=\sum_{g\in H\backslash S_{2m}/H}\dfrac{|HgH||H\cap gHg^{-1}|}{|S_{2m}|}\]
		\[=|H\backslash S_{2m}/H|\dfrac{|H|^2}{|S_{2m}|}=p(m)\dfrac{2^{2m}(m!)^2}{(2m)!}\]
		\[\sim p(m)\dfrac{2^{2m}\cdot 2\pi m(m/e)^{2m}}{\sqrt{2\pi\cdot 2m}(2m/e)^{2m}}=p(m)\sqrt{\pi m}.\]
		Since $p(m)\sim \dfrac{1}{4\sqrt{3}m}e^{\pi\sqrt{2m/3}}$ by Hardy-Ramanujan \cite{Hardy-Ramanujan}, the average is of super-polynomial growth, which is a sign that the density of good elements should be low.
	\end{remark}
	\subsection{Counting double cosets by character formula}\label{subsection-character formulas}
	Apart from the combinatorial methods in section \ref{subsection-double cosets counting}, there is also an applicable method of counting (self-inverse) double cosets by character formula. 
	\begin{prop}[J.S. Frame \cite{J.S.Frame}, Theorem A]\label{prop-J.S.Frame}
		The number of self-inverse double cosets of a finite group $G$ with respect to a subgroup $H\le G$ equals
		\[\sum_{\chi\in \Irr{G}, \langle\chi, \Ind_H^G\boldmath{1}_H\rangle\neq 0}FS(\chi),\]
		where the sum is over Frobenius-Schur indicators of irreducible characters occurring in the induced character of $G$ from the trivial character of $H$. Here for any character $\chi$ of $G$, 
		\[FS(\chi):=\dfrac{1}{|G|}\sum_{x\in G}\chi(x^2).\]
		Note that $\Ind_H^G\boldmath{1}_H$ is afforded by the permutation representation of $G$ through its action on the right cosets $H\backslash G$.
	\end{prop}
	\begin{proof}
		We follow the ideas of \cite{J.S.Frame}.
		
		First, we show that the number of self-inverse double cosets of $G$ with respect to $H$ is 
		$$\dfrac{\#\{g_ix^2=hg_i\mid g_i\in H\backslash G, x\in G, h\in H\}}{|G|}.$$
		(See Theorem 3.1 of \cite{J.S.Frame}.) It suffices to show that each self-inverse double coset corresponds to $|G|$ solutions to the equation
		\begin{equation}\label{equation-self-inverse}h(g_ixg_i^{-1})=(g_ixg_i^{-1})^{-1},\end{equation}
		which says that the inverse of $t=g_ixg_i^{-1}$ belongs to its own right coset. Each double coset $HgH$ decomposes into right cosets as
		\[HgH=\coprod_{y\in H/(g^{-1}Hg\cap H)}Hgy,\]
		hence each left coset $h'gH\subset HgH$ intersects with each right coset $Hgy$ at $h'g(g^{-1}Hg\cap H)y$, all of which have $d=|g^{-1}Hg\cap H|$ elements. In particular, the inverse of each right coset is a left coset, so it intersects with its own right coset at $d$ elements, which count as $d$ values of $t_i$. Summing over all right cosets in $HgH$, we get $[H: (g^{-1}Hg\cap H)]d=|H|$ solutions to (\ref{equation-self-inverse}) in $HgH$ if it is an self-inverse double coset. Varying the right cosets $g_i\in H\backslash G$, for each solution $(x_0,h_0)\in HgH\times H$ $hx=x^{-1}$ , we get solutions $(g_i^{-1}xg_i, h)$ to $hg_ixg_{i}^{-1}=(g_ixg_{i}^{-1})^{-1}$, which amount to $[G:H]|H|=|G|$ solutions. 
		
		Now let $G$ act on $H\backslash G$ by right multiplication and consider the corresponding permutation representation of $G$, which affords $\Ind_H^G\boldmath{1}_H$ by definition. Since the character value of a permutation representation on every element is the number of its fixed points, we get
		\begin{align*}&\dfrac{\#\{g_ix^2=hg_i\mid g_i\in H\backslash G, x\in G, h\in H\}}{|G|}\\
		=&\dfrac{1}{|G|}\sum_{x\in G}\Ind_H^G\boldmath{1}_H(x^2)\\
		=&FS(\Ind_H^G\boldmath{1}_H(x^2))\\
		=&\sum_{\chi\in \Irr{G}, \langle\chi, \Ind_H^G\boldmath{1}_H\rangle\neq 0}FS(\chi).
		\end{align*}		
	\end{proof}
	Next, we resort to an interesting result of Inglis-Richardson-Saxl \cite{Saxl} on multiplicity free decomposition of the permutation representation $\Ind_H^{S_{2m}}\boldmath{1}_H$. 
	\begin{prop}\label{prop-multiplicity free}
		Let $H=C(h_0), h_0=(1\ 2)(3\ 4)\cdots(2m-1\ 2m)$, then
		\[\Ind_H^{S_{2m}}\boldmath{1}_H=\bigoplus_{|\lambda|=m}S^{2\lambda},\]
		where $S^\nu$ for any partition $\nu$ denotes the Specht module (over $\Q$).  
	\end{prop}
	By Proposition \ref{prop-structure of double cosets}, the double cosets of $S_{2m}$ with respect to $H$ are all self-inverse for $x$ conjugate to $x^{-1}$ in $\Sym_\{2,4,\cdots,2m\}$. Also note that all irreducible representations of symmetric groups are of real type, i.e. $FS(\chi)=1$ for any $\chi\in \Irr S_{2m}$. Then Proposition \ref{prop-multiplicity free} and Proposition \ref{prop-J.S.Frame} show that the number of double cosets $H\backslash S_{2m}/H$ equals
	$$\sum_{\chi\in Irr{G}, \langle\chi, Ind_H^G\boldmath{1}_H\rangle\neq 0}FS(\chi)=\sum_{|\lambda|=m}FS(S^{2\lambda})=\sum_{|\lambda|=m}1=p(m),$$
	the partition number.
	\subsection{Structure of $H\cap gHg^{-1}$ and proof of Theorem \ref{thm-structure theorem}}\label{subsection-finer structure of centralizers}
	With the structure description of double cosets $H\backslash S_{2m}/H$, this section proves Theorem \ref{thm-structure theorem} using an idea of constructing bipartite graph automorphisms introduced by J.P. James \cite{J.P.James}.
	
	Let $\cG=(V,E)$ be a \textit{bipartite} graph (non-directed), i.e. its vertex set $V=V_1\coprod V_2$ is a disjoint union of two \textit{parties} $V_i, i=1,2$ and the edge set $E$ is a collection of (unordered) pairs $\{v_1,v_2\}, v_i\in V_1, i=1,2$. We allow one edge to be duplicated. A graph automorphism is a permutation of vertices that sends edges to edges. Denote $\Aut_b(\cG)$ the set of automorphisms preserving $V_i, i=1,2$. Suppose $\cG$ is $k$-regular, i.e. each vertex belongs to $k$ edges, then $|E|=kl$ for some positive integer $l$. Label the edges by integers between $1$ and $kl$. Define two $k$-partitions of $\{1,\cdots,kl\}$ as 
	\[\alpha_i=\{U_v, v\in V_i\}, i=1,2,\]
	in which $U_v=\{1\leq i\leq kl, v \text{ belongs to } i\}$, the set of all edges containing $v$. Then any automorphism of $\Aut_b(\cG)$ is a permutation of $\{1,\cdots, kl\}$ that preserves the two $k$-partitions $\alpha_1,\alpha_2$. Denote the group of such permutations $(S_{kl})_{\alpha_1,\alpha_2}$, then by definition $\Aut_b(\cG)\leq (S_{kl})_{\alpha_1,\alpha_2}$. 
	
	On the other hand, each permutation of $(S_{kl})_{\alpha_1,\alpha_2}$ is an automorphism of $\Aut_b(\cG)$. This is simply because each part of $\alpha_i$ (a $k$-subset of $\{1,\cdots.kl\}$) corresponds to a vertex in $V_i$, hence a permutation preserving $\alpha_i$ sends a vertex to a vertex, which also sends edges to edges by definition. We summarize Lemma 2.2 and 2.3 of \cite{J.P.James} as follows
	\begin{prop}\label{prop-bipartite automorphisms}
		$(S_{kl})_{\alpha_1,\alpha_2}\simeq\Aut_b(\cG)$.
	\end{prop}
	\begin{proof}[Proof of Theorem \ref{thm-structure theorem}]	The one-to-one correspondence $H\backslash S_{2m}/H$ was already established in Proposition \ref{prop-structure of double cosets}.	
	In application of Proposition \ref{prop-bipartite automorphisms} to our case, let $k=2, l=m$, the edges be $1,2,\cdots,2m$, and the two parties of vertices be $\alpha_1=\{\{1,2\},\cdots,\{2m-1,2m\}\}$ and $\alpha_2=\{\{g(1),g(2)\},\cdots,\{g(2m-1),g(2m)\}\}$ for any $g\in S_{2m}$. The edge $i$ connects two vertices (blocks $\{2k-1, 2k\}$'s) that contain $i$. Then by definition, $(S_{2m})_{\alpha_1,\alpha_2}=H\cap gHg^{-1}$. Recall that $H=C(h_0), h_0=(1\ 2)\cdots(2m-1\ 2m)$. By Proposition \ref{prop-structure of double cosets}, the structure of $H\cap gHg^{-1}$ depends only on those $g$ supported on even (or odd) numbers and their cycle type determined by partitions of $M'=\{2,4,\cdots,2m\}$. Hereinafter we denote a partition by $\lambda=\{1^{r_1}\cdots k^{r_k}\}$ which means $\lambda$ has $r_i$ parts equal to $i$ and by $N_\lambda=\sum_{i=1}^{k}r_i$ the number of parts of $\lambda$. If $g\in HxH$ for $x$ in the conjugacy class of $\Sym(M')$ with cycle type $\lambda$ , then the constructed bipartite graph $\cG$ has $N_\lambda$ connected components corresponding to parts of $\lambda$, i.e. cycles of $x$. For instance, the component corresponding to a part $k$ of $\lambda$, which may be expressed as the standard cycle $(2\ 4\cdots\ 2k)\in\Sym(M')$, looks like    
    
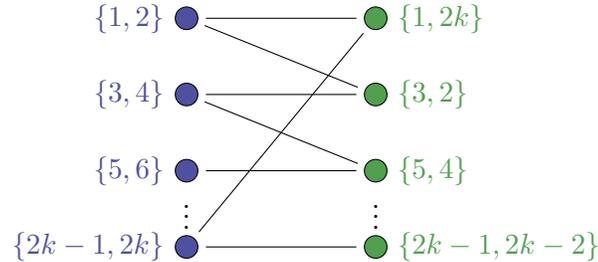
\[\begin{tikzpicture}[
    node distance = 7mm and 21mm,
    start chain = going below,
    V/.style = {circle, draw, 
    	fill=#1, 
    	inner sep=0pt, minimum size=3mm,
    	node contents={}},
    every fit/.style = {ellipse, draw=#1, inner ysep=-1mm, 
    	inner xsep=5mm},
    ]
    \node(n14)[V=myblue,on chain,
    label={[text=myblue]left:$\{1,2\}$}];
    \node(n13)[V=myblue,on chain,
    label={[text=myblue]left:$\{3,4\}$}];
    \node(n12)[V=myblue,on chain,
    label={[text=myblue]left:$\{5,6\}$}];
    \node(n11)[V=myblue,on chain,
    label={[text=myblue]left:$\{2k-1,2k\}$}];
    \node at ($(n12)!.5!(n11)$) {\vdots};
    
    \node(n24)[V=mygreen,right=0mm and 22mm of n14,
    label={[text=mygreen]right:$\{1,2k\}$}];
    \node(n23)[V=mygreen,right=0mm and 22mm of n13,
    label={[text=mygreen]right:$\{3,2\}$}];
    \node(n22)[V=mygreen,right=0mm and 22mm of n12,
    label={[text=mygreen]right:$\{5,4\}$}];
    \node(n21)[V=mygreen,right=0mm and 22mm of n11,
    label={[text=mygreen]right:$\{2k-1,2k-2\}$}];
    \node at ($(n22)!.5!(n21)$) {\vdots};

    \draw[-, shorten >=1mm, shorten <=1mm]
    (n14) edge (n24)   (n14) edge (n23)
    (n13) edge (n23)   (n13) edge (n22)
    (n12) edge (n22)   
    (n11) edge (n21)
    (n11) edge (n24);
    
    \end{tikzpicture}\]
	When unfolded, it becomes a $2k$-gon
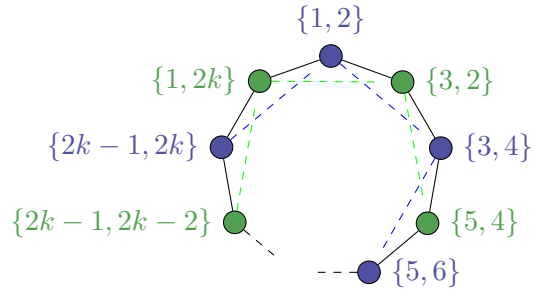
\[\begin{tikzpicture}[
V/.style = {circle, draw, 
	fill=#1, 
	inner sep=0pt, minimum size=3mm,
	node contents={}},
every fit/.style = {ellipse, draw=#1, inner ysep=-1mm, 
	inner xsep=5mm},
]    
\node (n1)[V=myblue,on chain,label={[text=myblue]above:$\{1,2\}$}]{};
\node (n2)[V=mygreen,position=-20:{\nodeDist} from n1,label={[text=mygreen]right:$\{3,2\}$}] {};
\node (n3)[V=myblue,position=-60:{\nodeDist} from n2,label={[text=myblue]right:$\{3,4\}$}] {};
\node (n4)[V=mygreen,position=-100:{\nodeDist} from n3,label={[text=mygreen]right:$\{5,4\}$}] {};
\node (n5)[V=myblue,position=-140:{\nodeDist} from n4,label={[text=myblue]right:$\{5,6\}$}] {};
\node (n6)[,position=-180:{\nodeDist} from n5,label={[text=myblue]right:{}}] {};
\node (n7)[V=mygreen,position=-220:{\nodeDist} from n6,label={[text=mygreen]left:$\{2k-1,2k-2\}$}] {};
\node (n8)[V=myblue,position=-260:{\nodeDist} from n7,label={[text=myblue]left:$\{2k-1,2k\}$}] {};
\node (n9)[V=mygreen,position=-300:{\nodeDist} from n8,label={[text=mygreen]left:$\{1,2k\}$}] {};
	\draw (n7) -- (n8) -- (n9) -- (n1) -- (n2) -- (n3) -- (n4) -- (n5);
	\draw[dashed,shorten >=5pt] (n5) -- (n6);
	\draw[dashed,shorten >=5pt] (n7) -- (n6);
	\draw[dashed,shorten >=5pt] 
	(n1) edge[blue] (n3)
	(n3) edge[blue] (n5)
	(n8) edge[blue] (n1)
	
	(n2) edge[green] (n4)
	(n9) edge[green] (n2)
	(n7) edge[green] (n9);
	\end{tikzpicture}\]
	
	Denote such a bipartite graph by $\cG_k$. Clearly as a proper subgroup of the automorphism group of the above $2k$-gon, i.e. $D_{2k}$, $\Aut_b(\cG_k)$ contains the automorphism group of the $k$-polygon with blue nodes (or equivalently the $k$-gon with green nodes) and dashed edges, i.e. $D_{k}$. Hence $\Aut_b(\cG_k)\simeq D_{k}$, the dihedral group with $2k$ elements. Any automorphism in $\Aut_b(\cG)$ can also permute components of the same size, i.e. those corresponding to cycles of the same length. Thus the above construction using bipartite graphs replicates the definition of wreath product with symmetric groups. Hence for any permutation $x\in\Sym(M')$ with cycle type $\{i^{r}\}$, by Proposition \ref{prop-bipartite automorphisms} we have the wreath product presentation 
	\[H\cap xHx^{-1}\simeq\Aut_b(\cG)\simeq \Aut_b{\cG_i}\wr S_r=D_{i}\wr S_r,\]
	In general for any $g\in HxH$ and $x\in \Sym(M')$ of cycle type $\lambda=\{1^{r_1}2^{r_2}\cdots k^{r_k}\}$,
	we get 
	\[H\cap gHg^{-1}\simeq \bigoplus_{i=1}^k D_{i}\wr S_{r_i},\]
	and in particular, 
	\[|H\cap gHg^{-1}|=\prod_{i=1}^k(2i)^{r_i}r_i!,\]
	simply by $|D_{i}\wr S_{r_i}|=|D_{i}|^{r_i}|S_{r_i}|$. This completes the proof.
	\end{proof}	
	Using Theorem \ref{thm-structure theorem} we can measure the double cosets as follows
	\begin{cor}\label{cor-double coset size}
			For any $g\in HxH$ with $x\in \Sym\{2,4,\cdots,2m\}$ with $x$ of cycle type $\lambda=\{1^{r_1}\cdots k^{r_k}\}$,
			\[|HgH|=|H|[H: H\cap gHg^{-1}]=(2^mm!)^2/(\prod_{i=1}^k(2i)^{r_i}r_i!).\]
	\end{cor}

		By Theorem 4.4.8 of James-Kerber \cite{James-Kerber}, the wreath product of a rational finite group with any symmetric group is also rational, hence Theorem \ref{thm-structure theorem} implies
		\begin{cor}\label{cor-rational group}
			All irreducible representations of $H\cap gHg^{-1}$ are realizable over $\Q$.
		\end{cor}
    \subsection{Some computational verification of Theorem \ref{thm-structure theorem}}\label{subsection-structure theorem verification}
    For convenience, we denote $g\sim\lambda$ for any $g\in S_{2m}$ and $\lambda$ a partition of $m$, if $g\in HxH$ with $x\in\Sym\{2,4,\cdots,2m\}$ of cycle type $\lambda$.
    
    For the simplest example, if $x\sim \{1^m\}$, then Theorem \ref{thm-structure theorem} gives 
	\[H\cap xHx^{-1}\simeq D_1\wr S_m=C_2\wr S_m,\]
	which coincides with Proposition \ref{prop-structure of H} because of $HxH=H$. 

        For $m=2$, $S_{4}$ as $p(2)=2$ double cosets, the nontrivial of which can choose a representative $x\sim \{2^1\}$, then Theorem \ref{thm-structure theorem} gives
        \[H\cap xHx^{-1}\simeq D_2\simeq K_4,\]
        which coincides with our computation by hand in Example \ref{exam-m=2}.

		For $m=3$, there are $3=p(3)$ double cosets in $H\backslash S_6/H$ with representatives $1, (4\ 5), (2\ 3)(4\ 5)$. Computed by \texttt{GAP} (the \textit{StructureDescription} function), we get
		\[H\cap (4\ 5)H(4\ 5)\simeq C_2\times C_2\times C_2\simeq D_1\times D_2,\]
		and
		\[H\cap (2\ 3)(4\ 5)H(2\ 3)(4\ 5)\simeq S_3\simeq D_3,\]
		where $D_i$ denotes the dihedral group with $2i$ elements and for convenience, we write $C_2$ as $D_1$. Note that $(4\ 5)\sim\{1^12^1\}$ and $(2\ 3)(4\ 5)\sim\{3^1\}$, the structure results by Theorem \ref{thm-structure theorem} coincide with computation by \texttt{GAP}.	
		
		For $m=4$, computed by \texttt{GAP} (the \textit{DoubleCosetRepsAndSizes} function), there are $5=p(4)$ double cosets in $H\backslash S_8/H$ with representatives 
		\[1, (6\ 7)\sim\{1^2 2^1\}, (4\ 5)(6\ 7)\sim\{1^13^1\}, (2\ 3)(6\ 7)\sim\{2^2\}, (2\ 3)(4\ 5)(6\ 7)\sim\{4^1\}.\]
		\texttt{GAP} gives the following structure description in coincidence with Theorem \ref{thm-structure theorem}
		\begin{align*}
		&H\cap(6\ 7)H(6\ 7)\simeq C_2\times C_2\times D_4\simeq (D_1\wr S_2)\times D_2,\\
		&H\cap (2\ 3)(6\ 7)H(2\ 3)(6\ 7)\simeq C_2^4\rtimes C_2\simeq D_2\wr S_2,\\
		&H\cap (4\ 5)(6\ 7)H(4\ 5)(6\ 7)\simeq D_{6}\simeq D_1\times D_3,\\
		&H\cap (2\ 3)(4\ 5)(6\ 7)H(2\ 3)(4\ 5)(6\ 7)\simeq D_4.
		\end{align*}		

        For $m=5$, by \texttt{GAP}, there are $7=p(5)$ double cosets in $H\backslash S_{10}/H$ with representatives 
        \[1,(8\ 9)\sim\{1^32^1\},(6\ 7)(8\ 9)\sim\{1^23^1\}, (4\ 5)(8\ 9)\sim\{1^12^2\},\]
        \[ (4\ 5)(6\ 7)(8\ 9)\sim\{1^14^1\}, (2\ 3)(6\ 7)(8\ 9)\sim\{1^12^13^1\}, (2\ 3)(4\ 5)(6\ 7)(8\ 9)\sim\{5^1\}.\]
        \texttt{GAP} gives the following structure description in coincidence with Theorem \ref{thm-structure theorem}
        \begin{align*}
        &H\cap (8\ 9)H(8\ 9)\simeq C_2\times C_2\times C_2\times S_4\simeq (D_1\wr S_3)\times D_2,\\
        &H\cap (6\ 7)(8\ 9)H(6\ 7)(8\ 9)\simeq D_4\times S_3\simeq (D_1\wr S_2)\times D_3,\\
        &H\cap (4\ 5)(8\ 9)H(4\ 5)(8\ 9)\simeq C_2\times(C_2^4\rtimes C_2)\simeq D_1\times (D_2\wr S_2),\\	
        &H\cap (4\ 5)(6\ 7)(8\ 9)H(4\ 5)(6\ 7)(8\ 9)\simeq C_2\times D_4= D_1\times D_4,\\
        &H\cap (2\ 3)(6\ 7)(8\ 9)H(2\ 3)(6\ 7)(8\ 9)\simeq C_2\times C_2\times S_3\simeq D_1\times D_2\times D_3,\\
        &H\cap (2\ 3)(4\ 5)(6\ 7)(8\ 9)H(2\ 3)(4\ 5)(6\ 7)(8\ 9)\simeq D_{5}.
        \end{align*}
        More computational verification by \texttt{GAP} for $m\geq6$ can also be checked. 
	\section{Count good elements}\label{section-counting good elements}
	With the structural results on $H\cap gHg^{-1}$, we are prepared to count good elements in $S_{2m}$. Recall that $g\in S_{2m}$ is \textit{good} if $|H\cap gHg^{-1}|=O(m^c)$ for some universal constant $c>0$. 
	
	\subsection{Counting with random permutation statistics}\label{subsection-distribution of size of centralizers}
	We show that the distribution of $|H\cap gHg^{-1}|$ happens to be the Ewens' distribution with bias $\theta=\frac{1}{2}$. By definition (see Example 2.19 of Arratia-Barbour-Tavar\'{e} \cite{Arratia}), the Ewens' distribution $\mathrm{ESF(\theta)}$ is the distribution equipped with the following probability density on partitions $\lambda=\{1^{r_1}\cdots k^{r_k}\}$ of $m$  
	\begin{equation}\label{ESF}P_{\theta}(\lambda)=\dfrac{m!}{\theta(\theta+1)\cdots(\theta+m-1)}\prod_{i=1}^{k}\left(\dfrac{\theta}{i}\right)^{r_i}\dfrac{1}{r_i!}.\end{equation}
	By Theorem \ref{thm-structure theorem} and Corollary \ref{cor-double coset size}, the distribution of $|H\cap gHg^{-1}|$ over $g\in S_{2m}$ is equivalent to the following probability density on partitions of  $m$, i.e. for any $x\in \Sym\{2,4,\cdots,2m\}$ of cycle type $\lambda$,
	\begin{equation}\label{equation-probability}P(\lambda)=\dfrac{|HxH|}{|S_{2m}|}=\dfrac{2^{2m}(m!)^2}{(2m)!\prod_{i=1}^k(2i)^{r_i}r_i!}
	=\dfrac{2^mm!}{\prod_{j=1}^{m}(2j-1)}\prod_{i=1}^{k}\left(\dfrac{\frac{1}{2}}{i}\right)^{r_i}\dfrac{1}{r_i!}\end{equation}
	\[=\dfrac{m!}{\prod_{j=1}^m(j-\frac{1}{2})}\prod_{i=1}^{k}\left(\dfrac{\frac{1}{2}}{i}\right)^{r_i}\dfrac{1}{r_i!},\]
	which is exactly $P_{\frac{1}{2}}(\lambda)$ as in (\ref{ESF}).
	
	This turns the study of distribution of $|H\cap gHg^{-1}|$ into study of Ewens' distribution $\rm ESF(\frac{1}{2})$. By Theorem 5.1 of \cite{Arratia}, as $m\rightarrow \infty$, $\mathrm{ESF}(\theta)$ point-wise converges to the joint distribution of independent Poisson distributions $(Z_1,Z_2,\cdots)$ on $\mathbb{N}^\infty$, where $Z_i\sim Po(\theta/i)$ for any $i\geq 1$ with $\rm Prob(Z_i=j)=e^{-\theta/i}\dfrac{(\theta/i)^j}{j!}$. However, the unmanageable errors appearing in \cite{Arratia} between Ewens' distributions and joint Poisson distribution make it inaccessible to calculate the tail distribution of $\rm ESF(\theta)$. In the next section, we use methods of analytic combinatorics to estimate the left tail $P(|H\cap gHg^{-1}|\leq m^c)$, i.e. the probability of good elements.
	\subsection{Left tail of Ewen's distribution}\label{subsection-Head of Ewen's distribution}
	First we define $|H\cap gHg^{-1}|$ as a random variable on partitions of $m$ i.e. let $f: \{\text{partitions of }m\}\rightarrow \R$ be $f(\lambda)=|H\cap gHg^{-1}|=\prod_{i=1}^k(2i)^{r_i}r_i!$ for any partition $\lambda=(1^{r_1}\cdots k^{r_k})$ of $m$ such that $g\in HxH$ for any $x\in \Sym\{2,4,\cdots,2m\}$ of cycle type $\lambda$, by Theorem \ref{thm-structure theorem}. 
	
	For any $a\in\R$, define $W_{a,m}:=\sum_{|\lambda|=m}f(\lambda)^{-a}$. Especially for $a=0$ we get the partition number $W_{0,m}=p(m)\sim \dfrac{1}{4\sqrt{3}m}e^{\pi\sqrt{2m/3}}$ and for $a=1$, $W_{1,m}=\dfrac{(2m)!}{2^{2m}(m!)^2}\sim \dfrac{1}{\sqrt{\pi m}}$ by section \ref{subsection-double coset decomposition}. Also note that $W_{a,m}$ strictly decreases as $a$ increases. In this notation we can write the distribution $P$ defined in (\ref{equation-probability}) as $P(\lambda)=W_{1,m}^{-1}f(\lambda)^{-1}$. 
	
	To estimate $P(f(\lambda)\leq m^c)$, i.e. the probability of good elements, we introduce the moment bound. For any nonnegative random variable $X$ from a sample space $\Omega$ to $\R_{\geq 0}$ with probability distribution $F$, define the $\alpha$-th moment for any $\alpha>0$ by
	\[M_X^\alpha:=\mathbb{E}({X^\alpha})=\int_\Omega X^\alpha(\omega)dF(\omega).\]
	Then by Markov's inequality, we have for any $C>0$,
	\[F(X>C)=F(X^\alpha>C^\alpha)\leq \dfrac{M_X^\alpha}{C^\alpha}.\]
	Since $\alpha$ is arbitrary, we get
	\begin{prop}[Moment bound]\label{prop-moment bound}
		For any $\alpha>0$ and nonnegative random variable $X$ with distribution $F$, 
		\[F(X\geq C)\leq \inf_{\alpha>0} \dfrac{M_X^\alpha}{C^\alpha},\  \forall C>0.\]
	\end{prop}
	Now for the distribution $P$ defined in (\ref{equation-probability}), the moment bound applied to $X=f^{-1}$ gives for any $c>0$,
	\begin{equation}\label{equation-left tail}P(f\leq m^c)=P(f^{-1}\geq m^{-c})\leq \inf_{\alpha>0}m^{c\alpha}{M_{f^{-1}}^\alpha}=\inf_{\alpha> 0}m^{c\alpha} W_{1,m}^{-1}W_{\alpha+1,m},\end{equation}
	since we have the expectation \[\mathbb{E}{f^{-\alpha}}=W_{1,m}^{-1}\sum_{|\lambda|=m}f(\lambda)^{-\alpha}f^{-1}(\lambda)=W_{1,m}^{-1}\sum_{|\lambda|=m}f(\lambda)^{-(\alpha+1)}=W_{1,m}^{-1}W_{\alpha+1,m}.\]

	Hence the task is to find appropriate estimate of $W_{\alpha+1,m}$ for $\alpha>0$. This is accessible through a hybrid method introduced by Flajolet et al \cite{Flajolet} which we present in section \ref{subsection-hybrid method}.
	\subsubsection{Generating function of $W_{\beta,m}$}\label{subsection-generating function}
	Before applying the hybrid method, it is necessary to introduce the following generating function for any $\beta\in\R$,
	\begin{equation}\label{equation-generating function}
	W_\beta(z)=\sum_{m\geq 0}W_{\beta,m}z^m=\sum_{m\geq 0}\sum_{|\lambda|=m}\dfrac{z^{r_1+2r_2+\cdots+kr_k}}{\prod_{i=1}^k(2i)^{r_i\beta}(r_i!)^{\beta}}=\prod_{i\geq 1}I_\beta(z^i/(2i)^\beta),
	\end{equation}
	where $I_\beta(z)=\sum_{j\geq 0}\dfrac{z^j}{(j!)^\beta}$ defines an entire function of exponential-like. For $\beta>0$, $W_\beta$ is an analytic function in the open unit disk of convergence radius $\geq1$ at the origin, since 
	\begin{equation}\label{equation-convergence radius upper bound}(W_{\beta,m})^{1/m}\leq W_{0,m}^{1/m}=p(m)^{1/m}\sim e^{\sqrt{m}/m}\rightarrow 1, \text{ as } m\rightarrow \infty.\end{equation}
	
	To further determine the convergence radius of $W_\beta(z), \beta>0$, we need a lower bound for $W_{\beta,m}$. For any $\alpha\in\R$, let $\mu_\alpha$ be the distribution on $\{\text{partitions of }m\}$ with $\mu_\alpha(\lambda)=W_{\alpha,m}^{-1}f(\lambda)^{-\alpha}$ for any partition $\lambda$ of $m$. For example, $\mu_0$ is the uniform distribution and $\mu_1$ is the distribution $P=P_{\frac{1}{2}}$ in the notation of Ewen's distribution defined in (\ref{ESF}).	For $0<\gamma<1$, $x^{1/\gamma}$ is a convex function, hence by Jensen's inequality (with expectation $\mathbb{E}_{\mu_\beta}$ over $\mu_\beta$), for any $\alpha,\beta\in\R$,
		\[\left(\mathbb{E}_{\mu_\beta}{f^{-\alpha}}\right)^{1/\gamma}\leq\mathbb{E}_{\mu_\beta}\left((f^{-\alpha})^{1/\gamma}\right),\]
		i.e.
		\[W_{\beta,m}^{-1}\sum_{|\lambda|=m}f^{-\alpha}(\lambda)f^{-\beta}(\lambda)=W_{\beta,m}^{-1}W_{\alpha+\beta,m}\]
		\[\leq\left(W_{\beta,m}^{-1}\sum_{|\lambda|=m}f^{-\alpha/\gamma}f^{-\beta}\right)^{\gamma}=W_{\beta,m}^{-\gamma}W_{\alpha/\gamma+\beta,m}^\gamma.\]
		Thus we get
		\begin{prop}\label{prop-Jensen's inequality}
			For any $\alpha,\beta\in\R$, $0<\gamma<1$, and $m\in\Z_+$,
			\[W_{\alpha+\beta,m}\leq W_{\beta,m}^{1-\gamma}W_{\alpha/\gamma+\beta,m}^\gamma.\]
		\end{prop}
		
		\begin{remark}\label{remark-upper bound}
			For $\beta=0$ and $0<\gamma=\alpha<1$, we get
			\[((\mathbb{E}_{\mu_0} f^{-\alpha}))^{1/\alpha}\leq \mathbb{E}_{\mu_0}(f^{-\alpha})^{1/\alpha}=\dfrac{W_{1,m}}{p(m)},\]
			i.e.
			\[\sum_{|\lambda|=m}f(\lambda)^{-\alpha}=W_{\alpha,m}\leq W_{1,m}^\alpha p(m)^{1-\alpha}.\]
			Let $\alpha=1-\frac{1}{\sqrt{m}}\frac{\sqrt{3}}{\pi\sqrt{2}}t\ln m$ for any $0<t<\frac{1}{2}$, by the asymptotics of $W_{1,m}$, $p(m)$ and $m^{\frac{\ln m}{\sqrt{m}}}=O(1)$, we get $W_{1,m}^\alpha=O(m^{-\frac{1}{2}})$ and $p(m)^{1-\alpha}=O(m^t)$, hence
			\[W_{\alpha,m}\leq O(m^{-\frac{1}{2}+t}).\]
			However, this bound is not sufficient for estimating the left tail in (\ref{equation-left tail}).
		\end{remark}
		\begin{remark}\label{remark-log convexity}
			Proposition \ref{prop-Jensen's inequality} is a log-convex constraint on $W_{\alpha,m}$, since
			\[(1-\gamma)\beta+\gamma(\alpha/\gamma+\beta)=\alpha+\beta.\]
			Especially for $\gamma=\frac{1}{2}$ we get
			\[W_{\alpha+\beta}\leq W_{\beta,m}^{\frac{1}{2}}W_{2\alpha+\beta,m}^{\frac{1}{2}},\]
			or 
			\[W_{2\alpha+\beta,m}\geq W_{\alpha+\beta}^2W_{\beta,m}^{-1}.\]	
		\end{remark}
	
	By the above remark, we can inductively prove 
	\begin{cor}\label{cor-convergence radius}
		For any $\beta>0$, $W_{\beta,m}^{1/m}\rightarrow 1$. Consequently, the convergence radius of $W_\beta(z)$ equals $1$. 
	\end{cor}
	\begin{proof}
		Acknowledging the upper bound (\ref{equation-convergence radius upper bound}), we need only to prove the lower bound. Assume for any $\beta\in(0,\gamma)$, $W_{\beta,m}^{1/m}\rightarrow 1$. By remark \ref{remark-log convexity}, for $\alpha=(\gamma-\beta)/2>0$,
		\[W_{\gamma,m}^{1/m}\geq (W^2_{\beta+\alpha,m})^{1/m}(W_{\beta,m})^{-1/m}\rightarrow 1^2\cdot 1=1,\]
		by induction. 
	\end{proof}
	
	\subsubsection{Exp-log schema for $W_{\beta}(z)$}\label{subsection-exp-log schema} Let $H_\beta(z)=\log(I_\beta(z))=\sum_{l\geq 1}h_{\beta,l}z^l$ ($h_{\beta,0}=0$ since $I_\beta(0)=1$), then (\ref{equation-generating function}) becomes
	\begin{equation}\label{equation-exp-log schema}
	W_\beta(z)=\exp\left(\sum_{i\geq 1}H_\beta\left(\dfrac{z^i}{2i^\beta}\right)\right)=\exp\left(\sum_{l\geq 1}\sum_{i\geq 1}h_{\beta,l}\dfrac{z^{il}}{(2i)^{\beta l}}\right)
	\end{equation}
	\[=\exp\left(\sum_{l\geq 1}\dfrac{h_{\beta,l}}{2^{\beta l}}Li_{\beta l}(z^l)\right),\]
	where $Li_\gamma(z)=\sum_{k\geq 1}\dfrac{z^k}{k^\gamma}$ is the polylogarithm for any $\gamma\in \C$. 
	
    Directly by definition, for $\gamma> 1$, $Li_\gamma(1)<\infty$ and $Li_\gamma(1)$ monotonically decrease to $1$ as
	    $\gamma\rightarrow\infty$. By Dirichlet's criterion, 
	    $\sum_{l>\lfloor\frac{1}{\beta}\rfloor}\dfrac{h_{\beta,l}}{2^{\beta l}}\left(Li_{\beta l}(1)-1\right)$ converges, and
	    \begin{equation}\label{equation-global order estimate 1}
	    \sum_{l>\lfloor\frac{1}{\beta}\rfloor}\dfrac{h_{\beta,l}}{2^{\beta l}}Li_{\beta l}(1)=\sum_{l>\lfloor\frac{1}{\beta}\rfloor}h_{\beta,l}2^{-\beta l}+O(1)
	    \end{equation}
	    \[=H_\beta(2^{-\beta})+O(1)=O(1).\]
	Hence if $\beta>1$, $W_{\beta}(z)$ is bounded in the unit disc $|z|<1$ (the convergence region), or of \textit{global order} $0$ in notation of the next subsection where we introduce the hybrid method in details. The boundedness also prevents us from directly using (Hardy-Littlewood-Karamata) Tauberian theorem to derive asymptotics for $W_{\beta,m}, \beta>1$. 
		    
    Additionally, the following result on singularities of polylogarithms is particularly helpful in this perspective.
	    \begin{lem}[Lemma 5 of \cite{Flajolet}]\label{lem-singularities of polylogarithms 1}
	    	For any $\gamma\in \C$, the polylogarithm $Li_\gamma(z)$ is analytically continuable to the slit plane $\C\smallsetminus\R_{\geq 1}$. Moreover, the singular expansion of $Li_\gamma(z)$ near the singularity $z=1$ for non-integer $\gamma$ is
	    	\begin{equation}\label{equation-polylogarithm singular expansion 1} Li_\gamma(z)=\Gamma(1-\gamma)\tau^{\gamma-1}+\sum_{j\geq 0}\dfrac{(-1)^j}{j!}\zeta(\gamma-j)\tau^j,\end{equation}
	    	where $\tau:=-\log z=\sum_{l\geq 1}\dfrac{(1-z)^l}{l}$, $\Gamma(z)$ is the gamma function and $\zeta(z)$ is the Riemann zeta function. For $m\in\Z_+$, 
	    	\begin{equation}\label{equation-polylogarithm singular expansion 2}Li_m(z)=\dfrac{(-1)^m}{(m-1)!}\tau^{m-1}(\log\tau-H_{m-1})+\sum_{j\geq 0,j\neq m-1}\dfrac{(-1)^j}{j!}\zeta(m-j)\tau^j,\end{equation}
	    	where $H_k$ is the harmonic number $1+1/2+\cdots+1/k$. 	
	    \end{lem}
	    In (\ref{equation-polylogarithm singular expansion 1}) (similar to (\ref{equation-polylogarithm singular expansion 2})), the first term is the singular part for $\gamma$ with real part $\rm Re\gamma\leq1$ and the regular remainder tends to $\zeta(\gamma)=Li_{\gamma}(1)$ if $\rm Re\gamma>1$, as $\tau\rightarrow 0$ (or $z\rightarrow 1$). The lemma indicates that for $0<\beta<1$, Tauberian theorem is also not directly applicable to $W_\beta(z)$, since $e^{a(-\log z)^{\beta-1}}\gg (1-|z|)^{-a}$ for any $a>0$, i.e. is of infinite global order. In section \ref{section-Tail of Ewens' distribution}, we will introduce asymptotics of coefficients of this type through a saddle point method  handled by E .M. Wright \cite{Wright}. 
	    
	    Note that $Li_\gamma(z^k)$ only has singularities at $k$-th roots $\xi_k$ of unity, the above lemma gives the corresponding singular expansion
	    	\begin{equation}\label{equation-singular expansion of polylogarithms at roots}
	    	Li_m(z^k)=\dfrac{(-1)^m}{(m-1)!k^{m-1}}(k\tau)^{m-1}(\log(k\tau)-H_{m-1})
	        \end{equation}
	        \[+\sum_{j\geq 0,j\neq m-1}\dfrac{(-1)^j}{j!k^j}\zeta(m-j)(k\tau)^j,\]
	        which becomes a series of $(1-z/\xi_k)$ by substitution
	        \[k\tau=-k\log(z/\xi_k)=\sum_{l\geq 1}\dfrac{k}{l}(1-z/\xi_k)^l.\]

	\subsection{Proof of Theorem \ref{thm-left tail} by hybrid method asymptotics for $W_{\beta,m}$}\label{subsection-hybrid method}
	We first introduce some necessary notions following Flajolet et al \cite{Flajolet}.
	\begin{defn}\label{defn-global order}
		The \textit{global order} of an analytic function $f(z)$ in the open unit disc, is a number $a\leq 0$ such that $|f(z)|=O((1-|z|)^a), \forall |z|<1$. 
	\end{defn}
	Since for any $\beta>1$, $W_\beta(z)$ is bounded in the unit disc, its global order is zero. It can be shown by Cauchy's integral formula that a function $f(z)$ of global order $a\leq 0$ has coefficients satisfying $[z^n]f(z)=O(n^{-a})$, see section 1.1 of \cite{Flajolet}.
	\begin{defn}\label{defn-log-power series}
		A \textit{log-power function} at $1$ is a finite sum of the form
		\[\sigma(z)=\sum_{k=1}^{r}c_k(-\log(1-z))(1-z)^{\alpha_k},\]
		where $\alpha_1<\cdots<\alpha_k$ and each $c_k$ is a polynomial. A log-power function at a finite set of points $Z=\{\zeta_1,\cdots,\zeta_m\}$, is a finite sum
		\[\Sigma(z)=\sum_{j=1}^m\sigma_j\left(\dfrac{z}{\zeta_j}\right),\]
		where $\sigma_j$ is a log-power function at $1$.
	\end{defn}
	Since $Li_0(z)=z/(1-z), Li_1(z)=-\log(1-z)$, a log-power function can be seen as approximation by combinations of these two polylogarithms. Asymptotics of coefficient of log-power functions are known, see Lemma 1 of \cite{Flajolet}.
	\begin{defn}\label{defn-smoothness}
		Let $h(z)$ be analytic in $|z|<1$ and $s$ be a nonnegative integer. $h(z)$ is said to be \textit{$\mathcal{C}^s$-smooth} on the unit disc, or of class $\mathcal{C}^s$, if for all $k=0,\cdots, s$, its $k$-th derivative $h^{(k)}(z)$ defined for $|z|<1$ admits a continuous extension to $|z|=1$.
	\end{defn}
	The smoothness condition relates to the coefficients of a function in an obvious way: if $h(z)=\sum_{n\geq 0}h_nz^n$ with $h_n=O(n^{-s-1-\delta})$ for some $\delta>0$ and $s\in \Z_{\geq 0}$, then it is $\mathcal{C}^s$-smooth. Conversely, we have the Darboux's transfer (Lemma 2 of \cite{Flajolet}): if $h(z)$ is $\cC^s$-smooth, then $h_n=o(n^{-s})$. By (\ref{equation-global order estimate 1}) and the easy differentiation formula $Li_{\gamma}'(z)=Li_{\gamma-1}(z)/z$, we can see that for any $\beta\geq 2$, $W_{\beta}(z)$ is at least $\cC^{\lfloor\beta\rfloor-2}$-smooth on the unit disc.
	
	\begin{defn}\label{defn-log-power extension}
		An analytic function $Q(z)$ in the open unit disc is said to admit a	\textit{log-power expansion} of class $\mathcal{C}^t$ if there exist a finite set of points $Z = \{\zeta_1,\cdots,\zeta_m\}$ on the unit circle $|z| = 1$ and a log-power function $\Sigma(z)$ at the points of $Z$ such
		that $Q(z)-\Sigma(z)$ is $\mathcal{C}^t$–smooth on the unit circle.    	
	\end{defn}
	By (\ref{equation-global order estimate 1}) and Lemma \ref{lem-singularities of polylogarithms 1}, $W_\beta(z)$ has non-trivial log-power expansion only for $\beta=1$ and for $0<\beta<1$ there exists no such expansion.
	\begin{defn}\label{defn-hybridization}
		Let $f(z)$ be analytic in the open unit disc. For $\zeta$ a point on the unit circle, we define the \textit{radial expansion} of $f$ at $\zeta$ with order $t\in\R$ as the smallest (in
		terms of numbers of monomials) log-power function $\sigma(z)$ at $\zeta$, provided it exists,
		such that
		\[f(z) = \sigma(z) + O((z-\zeta)^t),\]
		when $z=(1-x)\zeta$ and $x$ tends to $0^+$. The quantity $\sigma(z)$ is written
		\[asymp(f(z),\zeta,t).\]
	\end{defn}
	Now we are prepared to introduce the main theorem of the hybrid method.
	\begin{prop}[Theorem 2 of \cite{Flajolet}]\label{thm-hybrid}
		Let $f(z)$ be analytic in the open unit disc $D$, of
		finite global order $a\leq 0$, and such that it admits a factorization $f=P\cdot Q$, with
		$P,Q$ analytic in $D$. Assume the following conditions on $P$ and $Q$, relative to a
		finite set of points $Z = \{\zeta_1, . . . , \zeta_m\}$ on the unit circle $\partial D$:\\
		\textbf{D1}: The ``Darboux factor" $Q(z)$ is $\mathcal{C}^s$–smooth on $\partial D$ ($s\in \Z_{\geq 0}$).\\
		\textbf{D2}: The ``singular factor" $P(z)$ is analytically continuable to an indented domain of the form $\mathfrak{D}=\cap_{j=1}^m (\zeta_j\cdot\Delta)$, where a $\Delta$-domain is $\Delta(R,\phi):=\{z\in\C\mid |z|< R, \phi< arg(z-1)<2\pi-\phi, z\neq 1\}$ for some radius $R>1$ and angle $\phi\in (0,\frac{\pi}{2})$. For some non-negative real number $t_0$, it admits, at any $\zeta_j\in Z$, an asymptotic form (a log-power expansion of class $\mathcal{C}^{t_0}$)
		\[P(z)=\sigma_j(z/\zeta_j)+O((z-\zeta_j)^{t_0})\ (z\rightarrow \zeta_j, z\in\mathfrak{D}),
		\]
		where $\sigma_j(z)$ is a log-power function at $1$.\\
		\textbf{D3}: $t_0>u_0:=\lfloor\frac{s+\lfloor a\rfloor}{2}\rfloor$.\\
		Then $f$ admits radial expansions at every $\zeta_j\in Z$ with order $u_0=\lfloor\frac{s+\lfloor a\rfloor}{2}\rfloor$. The coefficients of $z^n$ of $f(z)$ satisfy:   
		\[[z^n]f(z)=[z^n]A(z)+o(n^{-u_0}),\]
		where $A(z):=\sum_{j=1}^m asymp(f(z),\zeta_j,u_0)$.
	\end{prop}
	
	Now we turn to approximating the coefficients of $W_\beta(z),\beta>1$, to the order $o(n^{-u_0})$ for some $u_0\in\Z_+$ which will be specified later as needed. We follow the hybrid method in close steps. 
	
	\subsubsection{Darboux factor}\label{subsection-Darboux factor}. By the theorem we should choose a Darboux factor of $\cC^s$-smooth for $s=2u_0$, note that the global order of $W_\beta(z)$ is zero. Provided the exp-log schema (\ref{equation-exp-log schema}), we can factorize $W_{\beta}(z)$ into
	\[W_\beta(z)=\exp\left(\sum_{l<\lfloor\frac{2u_0+2}{\beta}\rfloor}\dfrac{h_{\beta,l}}{2^{\beta l}}Li_{\beta l}(z^l)\right)\cdot\exp\left(\sum_{l\geq\lfloor\frac{2u_0+2}{\beta}\rfloor}\dfrac{h_{\beta,l}}{2^{\beta l}}Li_{\beta l}(z^l)\right)\]
	\[=e^{U(z)}\cdot e^{V(z)}.\]
	Since for $l\geq \lfloor\frac{2u_0+2}{\beta}\rfloor$, $\beta l\geq2u_0+2=s+2$, $Li_{\beta l}$ has all $k=0,\cdots, s$, its $k$-th derivatives admit a continuous extension onto the unit circle. Hence by Dirichlet's criterion as (\ref{equation-global order estimate 1}), $V(z)$ is $\cC^{s}$-smooth and we can take the Darboux factor as $Q(z)=e^{V(z)}$.
	
	\subsubsection{Singular factor}\label{subsection-singular factor}. Clearly we should take $P(z)=e^{U(z)}$ as the singular factor. $U(z)=\sum_{l<\lfloor\frac{2u_0+2}{\beta}\rfloor}\dfrac{h_{\beta,l}}{2^{\beta l}}Li_{\beta l}(z^l)$ as a truncation of the infinite sum, only has singularities at the $l$-th roots of unity for $l\leq \lfloor\frac{2u_0+2}{\beta}\rfloor-1$, by Lemma \ref{lem-singularities of polylogarithms 1}. This is to say $P(z)$ is analytically continuable to the intersection of $\Delta$-domains pointed at those roots, which form the set $Z$. Also the lemma readily shows that $P(z)$ admits the required asymptotic expansion to any order at each point of $Z$.
	
	Hence by the theorem, $W_\beta(z)$ for any $\beta>1$ admits a radial expansion at any point of $Z$ with the chosen order $u_0$ and the hybrid method could give us the wanted asymptotics for $W_{\beta,m}$ once the radial expansions is calculated explicitly at each singularity. To simplify calculation, we set $u_0=\lfloor\beta\rfloor$ so that we only need to consider the expansion at $l$-th roots of unity for $l\leq \left\lfloor\frac{2u_0+2}{\beta}\right\rfloor-1$, which evaluates as follows
	\[\left\lfloor\frac{2\lfloor\beta\rfloor+2}{\beta}\right\rfloor-1=\begin{cases}
	&2\quad\text{if }1<\beta\leq\frac{4}{3},\\
	&1\quad\text{if }\frac{4}{3}<\beta<2,\\
	&2\quad\text{if }\beta=2,\\
	&1\quad\text{if }\beta>2.
	\end{cases}\]
	In application, we mainly concern about the cases where $\beta\in\Z_{\geq 2}$ and $\beta\rightarrow 2^{-}$.
	\subsubsection{The expansion at $z=1,\beta\in\Z_{\geq 2}$}\label{subsection-expansion at 1} We first consider $\beta\in\Z_{\geq 2}$. Note that for any (real part) $\Re\gamma>1$, $\zeta(\gamma)=Li_{\gamma}(1)$ and
	\[W_\beta(1)=\exp\left(\sum_{l\geq 1}\dfrac{h_{\beta,l}}{2^{\beta l}}Li_{\beta l}(1)\right),\]
	by taking out $W_\beta(1)$ and using Lemma \ref{lem-singularities of polylogarithms 1} we get ($\tau=-\log z$)
	\begin{equation}\label{equation-log-power expansion}W_\beta(z)=W_\beta(1)\exp\left(\sum_{l\geq 1}\dfrac{h_{\beta,l}}{2^{\beta l}}\dfrac{(-1)^{\beta l}l^{\beta l-1}}{(\beta l-1)!}\tau^{\beta l-1}(\log\tau+\log l-H_{\beta l-1})\right)\end{equation}
	\[\cdot\exp\left(\sum_{l\geq 1}\dfrac{h_{\beta,l}}{2^{\beta l}}\sum_{j\geq 1, j\neq \beta l-1}\dfrac{(-1)^j}{j!}\zeta(\beta l-j)l^j\tau^j\right)\]
	\[=W_\beta(1)\exp\left(A_\beta(\tau)\log\tau+B_{\beta}(\tau)+\delta_\beta(\tau)\right)\]
	\[=W_\beta(1)+W_\beta(1)\sum_{n=1}^\infty\dfrac{1}{n!}\left(A_\beta(\tau)\log\tau+B_{\beta}(\tau)+\delta_\beta(\tau)\right)^n,\]
	in which $A_\beta,B_\beta,\delta_\beta$ are series of $\tau$ correspondingly. 
	
	Note that $\tau=-\log z=\sum_{l=1}^\infty\dfrac{(1-z)^l}{l}$, to approximate $W_\beta(z)$ by log-power functions at $z=1$ to the order $u_0=\lfloor\beta\rfloor$ is to approximate it to the order $O(\tau^\beta)$. Simply we have
	\[A_\beta(\tau)=\dfrac{(-1)^\beta h_{\beta,1}}{2^\beta(\beta-1)!}\tau^{\beta-1}+O(\tau^{2\beta-1}),\]
	\[B_{\beta}(\tau)=\dfrac{(-1)^\beta h_{\beta,1}}{2^\beta(\beta-1)!}(-H_{\beta-1})\tau^{\beta-1}+O(\tau^{2\beta-1}),\]
	\[\delta_\beta(\tau)=\sum_{j=1}^{\beta-1}\dfrac{(-1)^j}{j!}\tau^j\left(\sum_{l\geq 1, \beta l-1\neq j}\dfrac{h_{\beta,l}}{2^{\beta l}}\zeta(\beta l-j)l^j\right)+O(\tau^\beta)\]
	\[=\sum_{j=1}^{\beta-1}\dfrac{(-1)^jH_{\beta,j}}{j!}\tau^j+O(\tau^\beta),\]
	where $H_{\beta,j}=\sum_{l\geq 1, \beta l-1\neq j}\dfrac{h_{\beta,l}}{2^{\beta l}}\zeta(\beta l-j)l^j$ are convergent series.
	
	Hence in (\ref{equation-log-power expansion}), we only need to care about the following terms
	\[A_\beta(\tau)\log\tau,\ B_\beta(\tau),\  \sum_{n=1}^{\beta-1}\dfrac{1}{n!}\delta_\beta^n(\tau).\]
	We investigate the log-power expansion of these three terms separately.
	
	First we write $\log\tau$ as
	\[\log\tau=\log\left((1-z)\sum_{l=0}^\infty\dfrac{(1-z)^l}{l+1}\right)=\log(1-z)+\log\left(1+\sum_{l=1}^\infty\dfrac{(1-z)^l}{l+1}\right)\]
	\[=\log(1-z)+O(1-z).\]
   Then
   \[A_\beta(\tau)\log\tau=\dfrac{(-1)^\beta h_{\beta,1}}{2^\beta(\beta-1)!}\tau^{\beta-1}\log(1-z)+O(\tau^\beta)\]
   \[=\dfrac{(-1)^\beta h_{\beta,1}}{2^\beta(\beta-1)!}\left(\sum_{l=1}^\infty\dfrac{(1-z)^l}{l}\right)^{\beta-1}\log(1-z)+O(\tau^\beta)\]
   \[=\dfrac{(-1)^\beta h_{\beta,1}}{2^\beta(\beta-1)!}\left((1-z)^{\beta-1}+\dfrac{\beta-1}{2}(1-z)^\beta\right)\log(1-z)+O(\tau^\beta).\]

   The other two terms $B_\beta(\tau)$ and $\delta_\beta(\tau)$ do not involve $\log(1-z)$, hence for large enough $n$, do not contribute to $[z^n]W_\beta(z)$ by the following lemma
   \begin{lem}[Lemma 1 of \cite{Flajolet}]\label{lem-coefficients of log-powers}
   	The general shape of coefficients of a log-power function is computable by the two rules:
   	\[[z^n](1-z)^\alpha\sim\dfrac{1}{\Gamma(-\alpha)}n^{-\alpha-1}, \alpha\notin\Z_{\geq 0},\]
   	\[[z^n](1-z)^r(-\log(1-z))^k\sim(-1)^rk(r!)n^{-r-1}(\log n)^{k-1}, r\in\Z_{\geq 0}, k\in\Z_+.\]
   \end{lem}
   Note that $\Gamma(z)$ has poles at negative integers which makes the first formula in the lemma coincide with the obvious fact that $(1-z)^\alpha, \alpha\in\Z_{\geq0}$ do not contribute to asymptotics of coefficients eventually. Combined with the above calculation, we get 
   \begin{equation}\label{equation-coefficients of log-powers 1}
   [z^n]A_\beta(\tau)\log\tau=[z^n]\dfrac{(-1)^\beta h_{\beta,1}}{2^\beta(\beta-1)!}\left((1-z)^{\beta-1}+(\beta-1)(1-z)^\beta\right)\log(1-z)+o(n^{-\beta})
  \end{equation}
   \[=\dfrac{h_{\beta,1}}{2^\beta n^\beta}+o(n^{-\beta})=(2n)^{-\beta}+o(n^{-\beta}),\]
   the last of which recalls from $\sum_{l\geq 1}h_{\beta,l}z^l=\log(I_\beta(z))=\log\left(\sum_{j\geq 0}z^j/(j!)^\beta\right)$ that $h_{\beta,1}=1$ for any $\beta\in\R$. In general, the coefficients $h_{\beta,l}$ can be computed by Fa\`{a} di Bruno's formula. Hence we get the expansion for $W_\beta(z)$ at $z=1$ in this shape.
   \subsubsection{The expansion at $z=1, \beta>1, \beta\notin\Z_{\geq 0}$}\label{subsection-expansion at 1, beta goes to 2}
   By Lemma \ref{lem-singularities of polylogarithms 1} we get ($\tau=-\log z$)
   \begin{equation}\label{equation-Wbeta expansion at 1, beta goes to 2}
   W_\beta(z)=W_\beta(1)\exp\left(\sum_{l\geq 1}\dfrac{h_{\beta,l}l^{\beta l-1}}{2^{\beta l}}\Gamma(1-\beta l)(\tau)^{\beta l-1}\right)
   \end{equation}
   \[\cdot\exp\left(\sum_{l\geq 1}\dfrac{h_{\beta,l}}{2^{\beta l}}\sum_{j\geq 1}\dfrac{(-1)^j}{j!}\zeta(\beta l-j)l^j\tau^j\right)\]
   \[=W_\beta(1)\exp\left(A_\beta(\tau)+\delta_\beta(\tau)\right),\]
   in which (recall that $h_{\beta,1}=1$)
   \[A_\beta(\tau)=\dfrac{h_{\beta,1}}{2^\beta}\Gamma(1-\beta)\tau^{\beta-1}+O(\tau^{2\beta-1})\]
   \[=\dfrac{\Gamma(1-\beta)}{2^\beta}(1-z)^{\beta-1}+O((1-z)^\beta)\]
   and $\delta_\beta(\tau)$ involves only integer powers of $(1-z)$. Hence by Lemma \ref{lem-coefficients of log-powers}, we only need to concern about $A_\beta(\tau)$ and 
   \[[z^n]A_\beta(\tau)=\dfrac{\Gamma(1-\beta)}{2^\beta\Gamma(1-\beta)}n^{-\beta}+o(n^{-\beta})=\dfrac{1}{2^\beta n^\beta}+o(n^{-\beta}).\]
   \subsubsection{The expansion at $z=-1$}\label{subsection-expansion at -1}
   By Lemma \ref{lem-singularities of polylogarithms 1}, only $Li_{2\beta l}(z^{2l})$ in (\ref{equation-exp-log schema}) contribute singularities at $z=-1$, hence contribute to the asymptotics of $W_{\beta,n}$ to the order $O(n^{-2\beta})$ by \ref{subsection-expansion at 1} and \ref{subsection-expansion at 1, beta goes to 2}. 
   
   Thus combining \ref{subsection-Darboux factor}-\ref{subsection-expansion at 1, beta goes to 2} and \ref{subsection-expansion at -1}, we conclude from the hybrid method Theorem \ref{thm-hybrid} that
   \begin{prop}\label{prop-beta bigger than 1} For any $\beta>1$,
   	\[W_{\beta,m}=\dfrac{W_{\beta}(1)}{2^\beta m^\beta}+o(m^{-\beta}).\]
   \end{prop}
   \begin{remark}
   	We omit the calculation of $W_\beta(1)$ for now, but according to Proposition 4 of \cite{Flajolet}, it should be less than $4.26341/2^\beta$ for $\beta\geq 2$. Also note that by Stirling's formula $W_{1,m}\sim \frac{1}{\sqrt{\pi m}}$, an abrupt jump of order in $n$. This is caused by $W_\beta(1)\rightarrow \infty$ as $\beta\rightarrow 1$.
   \end{remark}
\begin{proof}[Proof of Theorem \ref{thm-left tail}]
   Now by the moment bound from Proposition \ref{prop-moment bound}, for $P$ defined as (\ref{equation-probability}) and $f$ the random variable on $\{\text{patitions of }m\}$ defined at the beginning of subsection \ref{subsection-Head of Ewen's distribution}, and for any $c>0,\alpha>0$, we have
   	\[P(f<m^c)\leq m^{c\alpha} W_{1,m}^{-1}W_{\alpha+1,m}=m^{c\alpha}\cdot O(m^{-1/2-\alpha})=O(m^{-1/2+(c-1)\alpha}).\]
   	In particular since for any $c>0$ there always exists $\alpha$ small enough such that $(c-1)\alpha<1/2$, we have 
   	\[P(f<m^{c})\rightarrow 0, \text{ as } m\rightarrow \infty,\]
   which proves Theorem \ref{thm-left tail}.
\end{proof}
	\section{Proof of Theorem \ref{thm-right tail} by Wright's expression}\label{section-Tail of Ewens' distribution}

	Again by Markov's inequality, for any $c>0, 0<\beta<1$ and expectation $\mathbb{E}_P$ on the probability measure $P$ defined in (\ref{equation-probability}),
	\begin{equation}\label{equation-right tail probability}P\left(f(\lambda)> m^c\right)=P\left(f^{1-\beta}(\lambda)>m^{c(1-\beta)}\right)\leq \dfrac{1}{m^{c(1-\beta)}}\mathbb{E}_P(f^{1-\beta})\end{equation}
	\[=m^{-c(1-\beta)}\sum_{|\lambda|=m}\left(\prod_{i=1}^k(2i)^{r_i}r_i!\right)^{1-\beta}\dfrac{2^{2m}(m!)^2}{(2m)!\prod_{i=1}^k(2i)^{r_i}r_i!}\]
	\[=m^{-c(1-\beta)}W_{1,m}^{-1}W_{\beta,m}.\]
	Only when $\beta$ tends $1$ could the above inequality give an appropriate bound for $P(f(\lambda)> m^c)$. The upper bound of $W_{\beta,m}$ in remark \ref{remark-upper bound} can only best possibly give 
	\[P(f(\lambda)> m^c)=O(1),\]
	for $\beta=1-\frac{1}{\sqrt{m}}\frac{\sqrt{3}}{\pi\sqrt{2}}t\log m$ and any $0<t<\frac{1}{2}$. Hence we need more precise asymptotics for $W_{\beta,m}, 0<\beta<1$.

    Let $\frac{1}{2}<\beta<1$, we can split $W_\beta(z)$ as of (\ref{equation-exp-log schema}) into 
    \begin{equation}\label{equation-Wbeta factorization, beta less than 1}W_\beta(z)=\exp\left(\sum_{l\geq 1}\dfrac{h_{\beta,l}}{2^{\beta l}}Li_{\beta l}(z^l)\right)=\exp\left(2^{-\beta}Li_\beta(z)\right)\cdot e^{V_\beta(z)},\end{equation}
    where $V_\beta(z)=\sum_{l\geq 2}\dfrac{h_{\beta,l}}{2^{\beta l}}Li_{\beta l}(z^l)$ and also note that $h_{\beta,1}=1$. By calculation using the hybrid method in subsection \ref{subsection-Head of Ewen's distribution}, it is clear that
    \begin{equation}\label{equation-hybrid asymptotics 1}[z^n]e^{V_\beta(z)}=O(n^{-2\beta}).\end{equation}
    For the first factor, by Lemma \ref{lem-singularities of polylogarithms 1}, we get ($\tau=-\log z$)
    \begin{equation}\label{equation-Wbeta expression 1}\exp\left(2^{-\beta}Li_\beta(z)\right)=\exp\left(2^{-\beta}\left(\Gamma(1-\beta)\tau^{\beta-1}+\zeta(\beta)+\delta_\beta(\tau)\right)\right),\end{equation}
    where
    \[\delta_\beta(\tau)=\sum_{j\geq 1}\dfrac{(-1)^j}{j!}\zeta(\beta-j)\tau^j.\]
    Similar to \ref{subsection-expansion at 1}, since $\delta_\beta(\tau)$ ($e^{\delta_\beta(\tau)}$) only involves integer powers of $(1-z)$, by Lemma \ref{lem-coefficients of log-powers} it does not contribute to the asymptotics of $[z^n]\exp\left(2^{-\beta}Li_\beta(z)\right)$ in order of $n$.
    Thus it is essential to approximate the coefficients of \[U_\beta(z)=\exp\left(2^{-\beta}\Gamma(1-\beta)(-\log z)^{\beta-1}\right).\] 
    
    Together with the factorization (\ref{equation-Wbeta factorization, beta less than 1}) and asymptotics (\ref{equation-hybrid asymptotics 1}), this gives
    \begin{equation}\label{equation-Wbeta expression 2}W_{\beta,m}=[z^m]W_\beta(z)=Ce^{2^{-\beta}\zeta(\beta)}\sum_{k=0}^m[z^n]U_\beta(z)[z^{m-n}]e^{V_\beta(z)}.\end{equation}
    Now we focus on the asymptotics of $[z^n]e^{2^{-\beta}\zeta(\beta)}U_\beta(z)$.      
    We notice that functions of same type with $U_\beta$ were already handled in 1930s by E. M. Wright \cite{Wright}.
    \begin{prop}[Wright's expansions, Theorem 5,6,7 of \cite{Wright}]\label{prop-Wright's expansions}
    	For any $a,b,c\in\C, a\neq 0$ and $\rho>0$, let
    	\[\chi(z)=\frac{z^{c}}{(-\log(z))^{b}} \exp \left(\dfrac{a}{(-\log(z))^{\rho}}\right),\]
    	and 
    	\[F(z)=\sum_{n=\lceil \Re c\rceil+1}^{\infty}(n-c)^{b-1}\phi(a(n-c)^\rho)z^n,\]
    	in which $\Re c$ is the real part of $c$ and 
    	\[\phi(z)=\sum_{l=0}^\infty\dfrac{z^l}{\Gamma(l+1)\Gamma(\rho l+b)}.\]
    	Then $F(z)$ forms the singular part of $\chi(z)$ and $G(z)=F(z)-\chi(z)$ is a regular function around $z=1$ where it behaves uniformly in terms of $a$ and $\rho$. Moreover, define the asymptotic expansion
    	\[H(z)=z^{1/2-b}e^{(1+1/\rho)z}\left(\sum_{j=0}^r\dfrac{(-1)^ja_j}{z^j}+O(\dfrac{1}{|z|^{r+1}})\right),\]
    	where the term $O(|z|^{r+1})$ and $a_j$ are uniformly bounded for $\rho>-1$, for example,
    	\[{a_{0}=\{2 \pi(\rho+1)\}^{-\frac{1}{2}}},\quad {a_{1}=\frac{12 b^{2}-12 b(\rho+1)+(\rho+2)(2 \rho+1)}{24(\rho+1)\{2 \pi(\rho+1)\}^{\frac{1}{2}}}}.\]
    	For $arg(z)=\xi, |\xi|\leq\pi-\epsilon$, let
    	\[Z=(\rho|z|)^{1 /(\rho+1)} e^{i\xi/(\rho+1)},\]
    	then $\phi(z)$ has the asymptotics (by a saddle point analysis which Wright did not perform in \cite{Wright} but in \cite{Wright Bessel})
    	\[\phi(z)=H(Z),\]
    	and the error term in $H$ depends on $\epsilon$.    	
    \end{prop}
    Since $V_\beta(z)$ is regular of global order $0$ at the singularity $z=1$, it does not contribute to asymptotics of coefficients (by Cauchy's integral formula). Thus we conclude from Proposition \ref{prop-Wright's expansions} that
    \begin{cor}\label{cor-Wright's expansions}
    Let $b=c=0, a=2^{-\beta}\Gamma(1-\beta)$ and $\rho=1-\beta$, then
    \[[z^n]U_\beta(z)=n^{-1}\phi(2^{-\beta}\Gamma(1-\beta)n^{1-\beta}).\]
    
    
    \end{cor}
    
    
    In particular $\xi=0$ (keeping notations of the above proposition), hence
    \[Z=\left((1-\beta)2^{-\beta}\Gamma(1-\beta)n^{1-\beta}\right)^{1/(2-\beta)}=\left(2^{-\beta}\Gamma(2-\beta)n^{1-\beta}\right)^{1/(2-\beta)}.\]
    Then
    \[[z^n]e^{2^{-\beta}\zeta(\beta)}U_\beta(z)=e^{2^{-\beta}\zeta(\beta)}\cdot n^{-1}H\left(\left(2^{-\beta}\Gamma(2-\beta)n^{1-\beta}\right)^{1/(2-\beta)}\right)\]
    \[=n^{-1}e^{2^{-\beta}\zeta(\beta)}\left(2^{-\beta}\Gamma(2-\beta)n^{1-\beta}\right)^{1/(4-2\beta)}\]
    \[\cdot\exp\left(\frac{2-\beta}{1-\beta}\left(2^{-\beta}\Gamma(2-\beta)n^{1-\beta}\right)^{1/(2-\beta)}\right)\cdot C\]
    \[=C'n^{-1}\cdot n^{\frac{1-\beta}{4-2\beta}}\exp\left(\frac{g(1-\beta)}{1-\beta}\right),\]
    where
    \[g(1-\beta)=(2-\beta)\left(2^{-\beta}\Gamma(2-\beta)n^{1-\beta}\right)^{1/(2-\beta)}+2^{-\beta}\zeta(\beta)(1-\beta),\]
    $C$ is bounded independent of $1-\beta$ and $n$, and $C'=C\cdot 2^{-\beta}\Gamma(2-\beta)\sim C/2$ as $\beta\rightarrow 1^{-}$.   
     
    Let $\epsilon=1-\beta\rightarrow 0^{+}$, then we can rewrite $g(1-\beta)$ as 
    \[g(\epsilon)=(1+\epsilon)\left(2^{\epsilon-1}\Gamma(1+\epsilon)n^{\epsilon}\right)^{\frac{1}{1+\epsilon}}+2^{\epsilon-1}\zeta(1-\epsilon)\epsilon.\]
    Hence to figure out the asymptotics of $W_{\beta,n}$ we need to compute the limit
    \[\lim_{\epsilon\rightarrow 0^{+}}\frac{g(\epsilon)}{\epsilon}=\frac{1}{2}+\lim_{\epsilon\rightarrow 0^{+}}\dfrac{\left(2^{\epsilon-1}\Gamma(1+\epsilon)n^{\epsilon}\right)^{\frac{1}{1+\epsilon}}+2^{\epsilon-1}\zeta(1-\epsilon)\epsilon}{\epsilon}\]
    \[=\frac{1}{2}+\dfrac{1}{2}\lim_{\epsilon\rightarrow 0^{+}}\dfrac{2^{\frac{\epsilon(1-\epsilon)}{1+\epsilon}}n^{\frac{\epsilon}{1+\epsilon}}\Gamma(1+\epsilon)^{\frac{1}{1+\epsilon}}+\zeta(1-\epsilon)\epsilon}{\epsilon}.\]
    First, the limit exists since $\zeta(1-\epsilon)\epsilon\rightarrow -1$ and then
    \[g(\epsilon)\rightarrow 1\cdot(2^{-1})^1+2^{-1}(-1)=0, \text{ as}\  \epsilon\rightarrow 0.\]
    Moreover, we have the Laurent series of $\zeta(s)$
    \[\zeta(s)=\frac{1}{s-1}+\sum_{n=0}^{\infty} \frac{(-1)^{n} \gamma_{n}}{n !}(s-1)^{n},\]
    where $\gamma_n$ are the Stieltjes constants and especially $\gamma_0$ is the Euler-Mascheroni constant. Thus we get
    \[\zeta(1-\epsilon)\epsilon=-1+\gamma_0\epsilon+\sum_{n=0}^{\infty} \frac{\gamma_{n}}{n !}\epsilon^{n+1}\sim -1+\gamma_0\epsilon+O(\epsilon^2),\]
    and we can rewrite the limit as
    \[\lim_{\epsilon\rightarrow 0^{+}}\frac{g(\epsilon)}{\epsilon}=\frac{1}{2}+\dfrac{1}{2}\lim_{\epsilon\rightarrow 0^{+}}\dfrac{2^{\frac{\epsilon(1-\epsilon)}{1+\epsilon}}n^{\frac{\epsilon}{1+\epsilon}}\Gamma(1+\epsilon)^{\frac{1}{1+\epsilon}}-1}{\epsilon}+\frac{\gamma_0}{2}\]
    \[=\frac{1+\gamma_0}{2}+\dfrac{1}{2}\lim_{\epsilon\rightarrow 0^{+}}\dfrac{n^{\frac{\epsilon}{1+\epsilon}}\Gamma(1+\epsilon)^{\frac{1}{1+\epsilon}}-2^{-\frac{\epsilon(1-\epsilon)}{1+\epsilon}}}{\epsilon}\]
    \[=\frac{1+\gamma_0}{2}+\dfrac{1}{2}\lim_{\epsilon\rightarrow 0^{+}}\dfrac{g_1(\epsilon)-g_2(\epsilon)}{\epsilon}.\]
    Easily $g_1(0)=g_2(0)=1$. Now we calculate their first derivatives at $0$,
    \[g_1'(\epsilon)=n^{\frac{\epsilon}{1+\epsilon}}\frac{1}{(1+\epsilon)^2}\log n\cdot\Gamma(1+\epsilon)^{\frac{1}{1+\epsilon}}\]
    \[+n^{\frac{\epsilon}{1+\epsilon}}\cdot\Gamma(1+\epsilon)^{\frac{1}{1+\epsilon}}\left(-\frac{1}{(1+\epsilon)^2}\log\Gamma(1+\epsilon)+\frac{1}{1+\epsilon}\dfrac{\Gamma'(1+\epsilon)}{\Gamma(1+\epsilon)}\right),\]
    hence 
    \[g_1'(0)=\log n-\gamma_0,\]
    note that $\Gamma(1)=1,\Gamma'(1)=-\gamma_0$. (Moreover, inductively we have estimate that $g_1^{(k)}(0)\sim (\log n)^k$.)
    
    \[g_2'(\epsilon)=2^{-\frac{\epsilon(1-\epsilon)}{1+\epsilon}}\log 2\cdot\left(1-\dfrac{2}{(1+\epsilon)^2}\right),\]
    hence
    \[g_2'(0)=-\log 2.\]
    Plugged into the limit, we get
    \[\lim_{\epsilon\rightarrow 0^{+}}\frac{g(\epsilon)}{\epsilon}=\frac{1+\gamma_0}{2}+\dfrac{1}{2}(g_1'(0)-g_2'(0))=\frac{1+\gamma_0}{2}+\dfrac{1}{2}(\log 2n-\gamma_0)\]
    \[=\dfrac{1+\log 2n}{2}.\]
    Moreover, we have (for $n\leq m$)
    \[\frac{g(\epsilon)}{\epsilon}-\dfrac{1+\log 2n}{2}=O\left(\sum_{k\geq 1}\dfrac{(\log n)^{k+1}}{(k+1)!}\epsilon^{k}\right),\]
    whereafter $O(*)$ means independent of $m$ and $1-\beta$.
    
    Finally we get 
    \[[z^n]e^{2^{-\beta}\zeta(\beta)}U_\beta(z)=O\left(n^{-1}\cdot n^{\frac{1-\beta}{4-2\beta}}\cdot\exp\left(\frac{g(1-\beta)}{1-\beta}\right)\right)\] 
    \[=O\left(n^{-1+\frac{1-\beta}{4-2\beta}}\exp\left(\frac{\log 2n}{2}+\log n\cdot O\left(\sum_{k\geq 1}\dfrac{((1-\beta)\log n)^{k}}{(k+1)!}\right)\right)\right)\]
    \[=O\left(n^{-\frac{1}{2}+\frac{1-\beta}{4-2\beta}+O\left(\sum_{k\geq 1}\dfrac{((1-\beta)\log n)^{k}}{(k+1)!}\right)}\right).\]
    Returning to (\ref{equation-Wbeta expression 2}) we finally get 
    \[W_{\beta,m}=[z^m]W_\beta(z)=O\left(\sum_{n=0}^m[z^n]e^{2^{-\beta}\zeta(\beta)}U_\beta(z)[z^{m-n}]e^{V_\beta(z)}\right)\]
    \[=O\left(m^{-\frac{1}{2}+\frac{1-\beta}{4-2\beta}+O\left(\sum_{k\geq 1}\dfrac{((1-\beta)\log m)^{k}}{(k+1)!}\right)}\right),\]
    note that $[z^{m-n}]e^{V_\beta(z)}=O\left(m^{-2\beta}\right)$.

    Returning to (\ref{equation-right tail probability}) and note that $W_{1,m}\sim (\pi m)^{-1/2}$, for any $c>0$ we get
    \[P(f>m^c)\leq m^{-c(1-\beta)}W_{1,m}^{-1}W_{\beta,m}\]
    \[=O\left(m^{\left(-c+\frac{1}{4-2\beta}\right)(1-\beta)+O\left(\sum_{k\geq 1}\dfrac{((1-\beta)\log m)^{k}}{(k+1)!}\right)}\right).\]
    
    For $\beta=1-\frac{t}{(\log m)^2}$ ($t$ constant) and $c>\frac{1}{2}+\log m$, the above term goes to zero as $m\rightarrow \infty$.
    This amounts to proving Theorem \ref{thm-right tail}.

\end{document}